\documentclass[12pt]{article}%
\usepackage{amsmath}
\usepackage{amsfonts}
\usepackage{amssymb}
\usepackage{graphicx,color}%
\usepackage{pstricks-add}
\usepackage{verbatim}
\providecommand{\U}[1]{\protect\rule{.1in}{.1in}}
\newtheorem{theorem}{Theorem}

\newtheorem{lemma}[theorem]{Lemma}

\newtheorem{proposition}[theorem]{Proposition}

\newenvironment{proof}[1][Proof]{\noindent\textbf{#1.} }{\ \rule{0.5em}{0.5em}}

\newcommand{\ab}{\partial_{\infty}}
\newcommand{\Hi}{\mathbb{H}}
\begin{document}

\title{On the asymptotic Plateau problem for CMC hypersurfaces in hyperbolic space}
\author{Jaime Ripoll, Miriam Telichevesky}
\maketitle

\begin{abstract}
Let $\mathbb{R}_{+}^{n+1}$ \ be the half-space model of the hyperbolic space
$\mathbb{H}^{n+1}.$ It is proved that if $\Gamma\subset\left\{  x_{n+1}%
=0\right\}  \subset\partial_{\infty}\mathbb{H}^{n+1}$ is a bounded $C^{0}$
Euclidean graph over $\left\{  x_{1}=0,\text{ }x_{n+1}=0\right\}  $ then,
given $\left\vert H\right\vert <1,$ there is a complete, properly embedded,
CMC $H$ hypersurface $\Sigma$ of $\mathbb{H}^{n+1}$ such that $\partial_{\infty
}\Sigma=\Gamma\cup\left\{  x_{n+1}=+\infty\right\}  .$ This result can be seen as
a limit case of the existence theorem proved by B. Guan and J. Spruck in
\cite{GS} on CMC $\left\vert H\right\vert <1$ radial graphs with prescribed
$C^{0}$ asymptotic boundary data.
\end{abstract}

\section{Introduction}
The following result was proved by B. Guan and J. Spruck in \cite{GS} (Theorem 4.8): 
\begin{theorem}
  \label{gs}
  Suppose $\Gamma$ is the boundary of a star-shaped $C^{0}$ domain in $\mathbb{R}^n$ and let $|H|<1$. Then there exists a unique hypersurface $\Sigma$ of constant mean curvature $H$ in $\Hi^{n+1}$ with asymptotic boundary $\Gamma$. Moreover, $\Sigma$ may be represented as the radial graph over the upper hemisphere $\mathbb{S}^n_+\subset\mathbb{R}^{n+1}$ of a function in $C^{\infty}(\mathbb{S}^n_+)\cap C^{0}(\overline{\mathbb{S}^n_+})$.
\end{theorem}

In the previous statement, the authors have used the half-space model of $\Hi^{n+1}$, namely, $$\Hi^{n+1}=\{(x,x_{n+1})\,|\, x\in \mathbb{R}^n, x_{n+1}>0\},$$ equipped with the metric $ds^2 = ds_E^2/x_{n+1}^2$, where $dS_E$ stands for the Euclidean metric in $\mathbb{R}^{n+1}$. In this model, $\{x_{n+1}=0\}$ is identified with the asymptotic boundary $\ab \Hi^{n+1}$ of $\Hi^{n+1}$, and hence their result guarantees the existence and uniqueness of CMC hypersurfaces with some prescribed asymptotic boundary data. 

\medskip

Before stating our main result, it is useful to write Theorem \ref{gs} in intrinsic terms. Using the conformal structure of hyperbolic spaces, arcs of circles and $n$-spheres on its asymptotic boundary do not depend on the chosen model (see first paragraph of Section \ref{hy}). The starshapedness property of $\Gamma$ is, under this approach, equivalent to the existence of two distinct points $p_{1},p_{2}\in\partial_{\infty}\mathbb{H}^{n+1},$ not belonging to $\Gamma,$ such that any arc of circle having $p_{1}$ and $p_{2}$ as ending points intersects $\Gamma$ at one and only one point.

In the present paper we study the degenerated case where $p:=p_{1}=p_{2}$. In this case, we assume that $\Gamma$ is contained in a pinched annulus bounded by two spheres that are tangent in $p$. Precisely: there are $(n-1)-$spheres $E_{1}$ and $E_{2}$ of $\partial_{\infty}\mathbb{H}^{n+1}$ which are tangent to $p\in E_{1}\cap E_{2}$ and $\Gamma\subset U_{1}\cap U_{2}$ where $U_{i}\subset\partial_{\infty}\mathbb{H}^{n+1}$ is the closure of the connected component of $\partial_{\infty}\mathbb{H}^{n+1}\backslash E_{i}$ that contains $E_{j},$ $i\neq j$ (note that this condition is the limit of the case in which $\Gamma$ is contained in an annulus of $\partial_{\infty}\mathbb{H}^{n+1}$ bounded by two spheres, condition which is trivially satisfied in Theorem \ref{gs} from the assumption that $p_{1} ,p_{2}\notin\Gamma$ and that $\Gamma$ is compact).  Moreover, we require that any circle passing through $p$ orthogonal to $E_{1}$ intersects $\Gamma$ at one and only one point (note also that this condition is the limit of the starshaped condition) and that $\Gamma$ is a compact topological hypersurface of $\ab \mathbb{H}^{n+1}$. Under these hypotheses, we prove that given $0\le |H|<1$, there exist a smooth CMC $H$ hypersurface $\Sigma$ such that $\ab \Sigma = \Gamma$. In the half-space model $\mathbb{R}^{n+1}_+$, if $p=\{x_{n+1}=\infty\}$, then $\Sigma$ is a horizontal graph over a totally geodesic hyperplane $\Hi^n$ containing $p$. In this case, $\Gamma$ is the horizontal graph of a bounded (in the Euclidean sense) continuous function defined on $\ab \Hi^n$.

\medskip

Our proof consists on showing the existence and uniqueness of a solution of the asymptotic Dirichlet problem defined on a totally geodesic hypersurface $\Hi^n$ of $\Hi^{n+1}$:

\begin{equation}\label{dp}\left\{ \begin{array}{l}Q_{H}\left(u\right):=\text{div}\left(\dfrac{\nabla u}{\sqrt{\gamma+|\nabla u|^{2}}}\right)-\dfrac{\gamma}{\sqrt{\gamma+|\nabla u|^{2}}}\left\langle \nabla u,\bar{\nabla}_{Y}Y\right\rangle =nH \text{ in } \Hi^n,\\ u|_{\ab \Hi^n} = \phi, \end{array}\right. \end{equation} which corresponds of finding a function $u$ such that its {\em parabolic} Killing graph is a CMC $H$ hypersurface with $\Gamma$ as asymptotic boundary. Here $Y$ is a parabolic Killing field in $\Hi^{n+1}$ orthogonal to $\Hi^n$, the function $\phi\in C^0(\ab \Hi^{n})$ has $\Gamma$ as its Killing graph and $\gamma=\langle Y,Y\rangle^{-1}$ (see Section \ref{killing} for the details). Furthermore, $\text{div}$ and $\nabla$ correspond to the divergent and gradient in $\Hi^n$ and $\overline{\nabla}$ is the Riemannian connection in $\Hi^{n+1}$.

\medskip

At this point it is important to remark that the radial property of $\Sigma$ in Theorem \ref{gs} is equivalent to $\Sigma$ being the \emph{hyperbolic} Killing graph of a function defined in a totally geodesic hypersurface of $\mathbb{H}^{n+1}$ and $\Gamma$ is the radial graph of some function $\phi\in C^0(\ab \Hi^n)$. Hyperbolic and parabolic Killing fields share nice properties that allow us to reprove Guan-Spruck's result in terms of Killing graphs in the same way that we prove our main result, see Section \ref{gsrevisited}.

\medskip 

We would like to remark that a natural approach to solve Dirichlet problems on unbounded domains, the whole space $\mathbb{H}^{n}$ in this case, is by taking the limit of a family of solutions obtained by solving the Dirichlet problem in an exhaustion of $\mathbb{H}^{n}$ by compact $C^{2,\alpha}$ domains $\Omega_{k},$ with boundary data being the restriction to $\partial\Omega_{k}$ of a $C^{2,\alpha}$ function of $\mathbb{H}^{n}$ that extends continuously the given asymptotic boundary data. However, to guarantee that the Dirichlet problem in each $\Omega_{k}$ is solvable the Killing cylinder 
over $\partial\Omega_{k}$ has to be $H-$mean convex. But it seems strongly to us that the existence of an exhaustion of $\mathbb{H}^{n}$ by domains $\Omega_{k}$ with $H-$mean convex Killing cylinders $C\left(  \partial \Omega_{k}\right)  $ may not exist in the case of {\em parabolic} Killing fields if $0<H<1$. To overcome this difficulty we change the approach and, instead of the exhaustion technique, we use Perron's method for the CMC equation of Killing graphs (Theorem \ref{perron}). The applicability of Perron's method in our situation is possible thanks to the recent results obtained by M. Dajczer, J. H. de Lira an the first author (Theorems 1 and 2 of \cite{DLR}. See also Theorems 1 and 2 stated below). With Perron's method, moreover, we obtain a simple proof of our main result that can be easily adapted to reprove Theorem \ref{gs}.

Finally, we point out that our proof is based on a geometric construction, which rests highly on the geometric structure of the hyperbolic space itself. This is specially illustrated on Proposition \ref{lbarrier}.

\section{Perron's method for the mean curvature equation of Killing graphs}\label{killing}

\qquad Let $N^{n+1}$ denote a $(n+1)$-dimensional Riemannian manifold. Assume that $N$ admits a Killing vector field $Z$ with no singularities, whose orthogonal distribution is integrable and whose integral lines are complete. Choose an integral hypersurface $M^{n}$ of the orthogonal distribution. It is easy to see that $M$ is a totally geodesic submanifold of $N$. Denote by  $\Psi_{Z}\colon\mathbb{R}\times N\rightarrow N$ the flow of $Z$. Notice that $\gamma=1/\left\langle Z,Z\right\rangle \>$can be seen as a function in $M$ since $Z\gamma=0$ by the Killing equation. Moreover, the solid cylinder $\Psi(\mathbb{R}\times M)$ with the induced metric has a warped product Riemannian structure $M\times_{\rho}\mathbb{R}$ where $\rho=1/\sqrt{\gamma}$. 

Given a function $u$ on an open subset $\Omega$ of $M$ the associated \emph{Killing graph} is the hypersurface $$\mbox{Gr}_{Z}(u)=\{\Psi(u(x),x) \,\vert\,x\in\Omega\}.$$ It is shown in \cite{DHL} that $\mbox{Gr}(u)$ has CMC $H$ if and only if $u\in C^{2}(\Omega)$ satisfies
\begin{equation}
Q_{H}\left(  u\right)  =\text{div}\Bigg(\frac{\nabla u}{\sqrt{\gamma+|\nabla u|^{2}}}\Bigg)-\frac{\gamma}{\sqrt{\gamma+|\nabla u|^{2}}}\left\langle \nabla u,\bar{\nabla}_{Z}Z\right\rangle =nH \label{pde}
\end{equation}
where 
$H$ is computed with respect to the orientation of $\mbox{Gr}_{Z}(u)$ given by the unit normal vector $\eta$ such that $\left\langle \eta,Z\right\rangle \leq0$. Also $\bar{\nabla}$ denotes the Riemannian connection of $N.$

\vspace{1ex}

Given $o\in\Omega$ let $r>0$ be such that $r<i(o)$ where $i(o)$ is the injectivity radius of $M$ at $o$. We denote by $B_{r}(o)$ the geodesic ball centered at $o$ and radius $r$ which closure is contained in $\Omega$. Theorems \ref{dlr} and \ref{dhl} below are proved in \cite{DLR}: 

\begin{theorem}\label{dlr}
{\hspace*{-1ex}}\textbf{. }\label{korevaar} Let $u\in C^{3}(B_{r}(o))$ be a
solution of the mean curvature equation (\ref{pde}). Then, there exists a
constant $L=L(u(o),r,\gamma,H)$ such that $|\nabla u(o)|\leq L$.
\end{theorem}

We remark that this result has been extended by J-b. Casteras and the first author to CMC Killing submersions (Theorem 4 of \cite{CR}).

\begin{theorem}
\label{dhl}{\hspace*{-1ex}}\textbf{. } Let $\Omega\subset M$ be a bounded $C^{2,\alpha}$ domain and $\Gamma=\partial\Omega.$ Denote by $$C\left(  \Gamma\right)  :=\left\{  \Psi\left(  t,p\right)  \,\vert\,p\in\Gamma,\,t\in\mathbb{R}\right\} $$ the Killing cylinder over $\Gamma$. Assume that the mean curvature $H_{\Gamma}\ $of $C\left(  \Gamma\right)  $ is nonnegative with respect to the inner orientation and that $$\inf_{\Omega}\operatorname*{Ric}\nolimits_{M}\geq-n\inf_{\Gamma}H_{\Gamma}^{2}.$$ Let $H\geq0$ be such that $\inf_{\Gamma}H_{\Gamma}\geq H$. Then, given $\varphi\in C^{0}(\Gamma)$ there exists a unique function $u\in C^{2,\alpha}(\Omega)\cap C^{0}(\bar{\Omega})$ whose Killing graph has mean curvature $H$ with respect to the unit normal vector $\eta$ to the graph of $u$ satisfying $\left\langle \eta,Z\right\rangle \leq0$ and $u|_{\Gamma}=\varphi$.
\end{theorem}

\bigskip

Recall that $\sigma\in C^{0}(M)$ is a subsolution for $Q_{H}$ in $M$ if, given a bounded subdomain $\Lambda\subset M,$ if $u\in C^{0}\left(\overline{\Lambda}\right)  \cap C^{2}\left(\Lambda\right)  $ satisfies $Q_{H}\left(  u\right)  =0$ in $\Lambda$ and $\sigma|_{\partial\Lambda}\leq u|_{\partial\Lambda}$ then $\sigma\leq u.$ Supersolution is defined by replacing ``$\leq$'' by ``$\geq$''.

\medskip

The validity of Perron's method can be proved for general Killing graphs and general domains of a Riemannian manifold. However, we state and prove it only on the case that the domain is a whole Cartan-Hadamard manifold and the boundary data is prescribed at infinity, case that we are interested in this paper.

\medskip

Let $N$ be a Cartan-Hadamard manifold, that is, $N$ simply connected with sectional curvature $K_{N}\leq 0.$ Then it is well defined the asymptotic boundary $\partial_{\infty}N$ of $N$, a set of equivalence classes of geodesic rays of $N$. We recall that two geodesic rays $\alpha,\beta:\left[0,\infty\right)  \rightarrow N$ are in the same class if there is $C>0$ such that $d(\alpha(t),\beta(t))\leq C$ for all $t\in\left[  0,\infty\right)  $ where $t$ is the arc length. With the so called cone topology $\overline {N}:=N\cup\partial_{\infty}N$ is a compactification of $N$ (see \cite{BO})$.$ Since $M$ is totally geodesic in $N,$ $M$ is also a Cartan-Hadamard manifold and we may naturally consider $\partial_{\infty}M$ as a subset of $\partial_{\infty}N.$ The equivalence relation between geodesics and convergent sequences are preserved under isometries, therefore the action of $\Psi$ on $N$ extends to the compactification $\overline{N}$ of $N$ and the extended action is continuous.

Taking into account the solvability of the Dirichlet problem \eqref{dp} in small balls (Theorem \ref{dhl} above), and the maximum principle, standard arguments (See Section 2.8 of \cite{GT}) prove the following  result:

\begin{lemma}
\label{comp}If $v,w\in C^{0}\left(\overline{M}\right)$  are respectively a subsolution and a supersolution of of $Q_{H}$ in $M$ such that $v|_{\partial_{\infty}M}\leq w|_{\partial_{\infty}M}$ then $v\leq w$ in $\overline{M}.$
\end{lemma}


\medskip We next prove Perron's method, which proof is also well known. For completeness, we give some details.

\begin{theorem}
\label{perron}Let $N$ be a Cartan-Hadamard manifold. Let $\phi\in C^{0}\left(\partial_{\infty}M\right)  $ and $H\in\mathbb{R}$ be given$.$ Assume that there are a subsolution $\sigma\in C^{0}\left(  \overline{M}\right)  $ and a supersolution $w\in C^{0}\left(  \overline{M}\right) $ of $Q_{H}$ in $M$ such that $\sigma|_{\partial_{\infty}M}\leq\phi\leq w|_{\partial_{\infty}M}.$ Set
$$S_{\phi}=\left\{  v\in C^{0}\left(  \overline{M}\right)\,\vert\, v\text{ is a subsolution of }Q_{H}\text{ such that }v|_{\partial_{\infty}M}<\phi\right\}$$
and
$$u_{\phi}(x):=\sup_{v\in S_{\phi}}v(x),\text{ }x\in\overline{M}.$$
Then $u_{\phi}\in C^{\infty}\left(  M\right)  $ and $Q_{H}\left(  u_{\phi}\right)  =0.$
\end{theorem}

\begin{proof}
From the hypothesis and Lemma \ref{comp} it follows that $u_{\phi}$ is a well defined function on $\overline{M}.$ Given $x\in M,$ let $B_{r}(x)\subset M$ be the open geodesic ball centered at $x$ and with radius $r>0$ in $M.$ From the Hessian comparison theorem, for $r=r_{x}$ sufficiently small, if 
$$C_{r_{x}}:=\left\{  \Psi\left(  t,p\right)\,\vert\,p\in\partial B_{r_{x}}(x),\text{ }t\in\mathbb{R}\right\},$$
then $H_{x}\geq\left\vert H\right\vert ,$ where $H_{x}$ is the mean curvature of $C_{r_{x}}$ with respect to the inner orientation$\ $and, clearly, $\inf_{B_{r_{x}}}\operatorname*{Ric}\nolimits_{M}\geq-n\inf_{C_{r_{x}}}H_{x}^{2}.$

Now, take a sequence $s_{m}\in C^{0}\left(  \overline{M}\right)  $ of subsolutions such that $s_{m}|_{\partial_{\infty}M}\leq\phi$ and $\lim _{m}s_{m}(x)=u_{\phi}(x).$ 
By Theorem \ref{dhl} there is a solution $w_{m,x}\in C^{\infty}\left(  B_{r_{x}}(x)\right)  $ of $Q_{H}=0.$ We now define the CMC $H$ lift $t_{m}\in C^{0}\left(  \overline{M}\right)  $ of $s_{m}$ by setting 
$$t_{m}(y)=\left\{ \begin{array} [c]{ll}s_{m}(y) & \text{if }y\in M\backslash B_{r_{x}}\left(  x\right) \\ w_{m,x}(y) & \text{if }y\in B_{r_{x}}\left(  x\right).\end{array} \right.$$
Using Theorem \ref{korevaar} it follows that $w_{m,x}$ contains a subsequence converging uniformly on compact subsets of $B_{r_{x}}(x)$ to a solution $w_{x}\in C^{\infty}\left(  B_{r_{x}}(x)\right)  $ of $Q_{H}=0.$ As in \cite{GT}, Section 2.8, we may prove that $w_{x}=u_{\phi}|_{B_{r_{x}}(x)}$. This proves that $u_{\phi}\in C^ {\infty}\left(  M\right)  $ and satisfies $Q_{H}\left(  u_{\phi}\right)  =0.$
\end{proof}

\

\



\section{\label{hy}Asymptotic Plateau's problem for CMC hypersurfaces in $\mathbb{H}^{n+1}$}

\qquad Let $\mathbb{H}^{n+1},$ $n\geq2,$ be the hyperbolic space with sectional curvature $-1$. It is well known that there is a conformal diffeomorphism $F$ between $\overline{\mathbb{H}}^{n+1}$ and the unit closed ball $\overline{B}$ of $\mathbb{R}^{n+1}$ such that $F|_{\mathbb{H}^{n+1}}:\mathbb{H}^{n+1}\rightarrow B$ is an isometry, and any Killing vector field in $B$ is the restriction of a conformal vector field of $\mathbb{R}^{n+1}$ that leaves $B$ invariant (see \cite{C}). 

Furthermore, spheres at the asymptotic boundary are well defined. A $m-$dimensional sphere $E$ of $\partial_{\infty}\mathbb{H}^{n+1}$ is defined as $E=F^{-1}(\widetilde{E})$ where $\widetilde{E}\subset\partial B$ is an usual $m-$dimensional sphere, that is, the intersection between $\partial B$ and a $\left(  m+1\right)  -$dimensional linear subspace of $\mathbb{R}^{n+1}.$ If $m=1$ then $E$ is a circle and if $m=n-1$ then $E$ is a hypersphere of $\ab \mathbb{H}^{n+1}$, wich from now on we will call simply ``sphere'', to avoid confusion with hyperspheres of hyperbolic space, that have spheres as asymptotic boundary.

From now on, we say that a set $A\subset \ab \Hi^{n+1}$ is {\em between} two spheres $E_1,E_2\subset \Hi^{n+1}$ if $A\subset U_1\cap U_2$, where $U_i$ is the connected component of $\ab \Hi^{n+1}\setminus E_i$ that contains $E_j,\,j\neq i$.

\medskip

With the previous remarks, we are ready to write our main result.

\begin{theorem}\label{par} Let $p\ $be a point in the asymptotic boundary $\partial_{\infty}\mathbb{H}^{n+1}$ of $\mathbb{H}^{n+1}$ and $\Gamma\subset\partial_{\infty}\mathbb{H}^{n+1}$be\ a compact embedded topological hypersurface passing through $p.$ Assume that there are two spheres $E_{1}$ and $E_{2}$ of $\partial_{\infty}\mathbb{H}^{n+1}$ which are tangent to $p\in E_{1}\cap E_{2}$ and such that $\Gamma$ is contained between $E_{1}$ and $E_{2},$ and, moreover, that any circle of $\partial_{\infty}\mathbb{H}^{n+1}$ passing through $p$ orthogonal to $E_{1}$ intersects $\Gamma$ at one and only one point.

Then, given $\left\vert H\right\vert <1,$ there exists a unique properly embedded, complete $C^{\infty}$ hypersurface $\Sigma$ of $\mathbb{H}^{n+1}$ with CMC $H$ such that $\partial_{\infty}\Sigma=\Gamma$ and $\overline{\Sigma}=\Sigma\cup\Gamma$ is a compact embedded topological hypersurface of $\overline{\mathbb{H}}^{n+1}.$

Moreover, $\Sigma$ is a parabolic graph, that is, there exists a totally geodesic hypersurface $\mathbb{H}^{n}$ of $\mathbb{H}^{n+1}$, a Killing field $Y$ which orbits in $\mathbb{H}^{n+1}$ are horocycles orthogonal to $\mathbb{H}^{n}$and in $\partial_{\infty}\mathbb{H}^{n+1}$ are the circles orthogonal to $E_{1}$ at $p,$ and a function $u\in C^{\infty}\left(  \mathbb{H}^{n}\right) \cap C^{0}\left(  \overline{\mathbb{H}}^{n}\backslash\left\{  p\right\}\right)  $ such that $\Sigma={\rm{Gr}}_{Y}(u)$.

\end{theorem}

Before proving Theorem \ref{par}, we present some special elements of $S_{\phi}$ (see notation of Theorem \ref{perron}) that work as lower barriers. 

\begin{proposition}\label{lbarrier}
Given $q$ in the connected component of $\ab\Hi^{n+1}\setminus \Gamma$ that contains $E_1=\ab \Hi^n$, there exists a function $w\in S_{\phi}$ such that the asymptotic boundary of ${\rm{Gr}}_Y(w)$ separates $q$ and $\Gamma$.  
\end{proposition}

\begin{proof}
Let $S\subset \mathbb{H}^{n+1}$ be the biggest totally geodesic hypersphere such that $\ab S$ is contained in the connected component of $\ab \Hi^{n+1}\setminus \Gamma$ that contains $E_1$. To do more easily some computations, we work on the upper half-space model $\mathbb{R}_+^{n+1}$; in this model we assume without loss of generality that $E_1=\{x_1=0,x_{n+1}=0\}$, $q=(l,0,\dots,0)$ for some $l>0$ and $S$ is the Euclidean (upper hemisphere of the) hypersphere centered at $q$ with radius $1$. In order to simplify the notation, we will write $(t,x,y)$, $t\in \mathbb{R}$, $x\in \mathbb{R}^{n-1}$ and $y>0$ as coordinates of any point. Hence $q=(l,0,0)$.

The choice of $S$ together with the fact that $\Gamma$ is a Killing graph imply that $\Gamma$ do not intersect the asymptotic boundary neither of the Killing (finite) cylinder $$C=\{(t,x,\sqrt{1-|x|^2})\,|\,t\in[0,l],\,|x|\le 1\}$$ nor of the upper hemisphere $$S^+=\{(t,x,y)\,|\,l\le t \le l+1, |x|\le 1, 0<y\le 1 \}.$$

We proceed with a ``stacking'' of portions of totally geodesic hyperspheres $H_k$ centered at the line $\{(t,0,\dots,0)\,|\,t\in[0,l]\}$ with decreasing radius and also with the condition that the (Euclidean) ``parallel $\alpha$'' of one of them coincides with the (Euclidean) ``parallel $\beta:=\alpha/2$'' of the next one. Here $\alpha\in(0,\pi/2)$ will be choosen small enough in order to obtain a graph separating $q$ from $\Gamma$. Precisely, we choose $\alpha$ such that 
\begin{equation}
l<\cos \beta \frac{\sin \alpha - \sin \beta}{\cos \beta - \cos \alpha} <l+1.  
\label{alpha}\end{equation}

 This choice is possible since the right-hand-side of the first inequality above is a continuous function on $\alpha$ that goes to $\infty$ as $\alpha$ goes to zero.

We define inductively the Euclidean center $C_k:=(t_k,0,\dots,0)$ and radii $R_k$ of the hyperspheres $H_k$:
$$R_0 =1, R_k=\frac{\cos \alpha}{\cos \beta}R_{k-1}, k\ge 1$$ (as an immediate consequence, $R_k=(\cos \alpha/\cos \beta)^k$) and
$$t_0 = -\sin \beta, t_k = t_{k-1} + R_{k-1}\sin\alpha - R_k \sin \beta.$$

The right upper hemispheres of $H_k$ are the Euclidean graphs of $$v_k: D_k\to \mathbb{R}, v_k(0,x,y) = t_k + \sqrt{R_k^2 - |(x,y)|^2},$$ where $D_k:=\{(0,x,y)\in \Hi^n\,|\,|(x,y)|\le R_k, y>0)\}$.

Notice that $D_0\supsetneq D_1 \supsetneq D_2 \subsetneq \dots$. A straithfoward computation gives that $v_k(P) > v_{k-1}(P)$ if and only if $$P\in W_k:=\{(0,x,y)\,|\,|(x,y)|<R_k\cos \beta = R_{k-1}\cos \alpha, y>0\},$$ that is, the upper hemispheres of $H_{k-1}$ and $H_k$ intersect exactly in the parallels of angle $\alpha$ and $\beta$, respectively, and besides $H_k$ is below $H_{k-1}$ in $D_k\setminus W_k$.

\medskip

\centerline{\fbox{\includegraphics[scale=0.17]{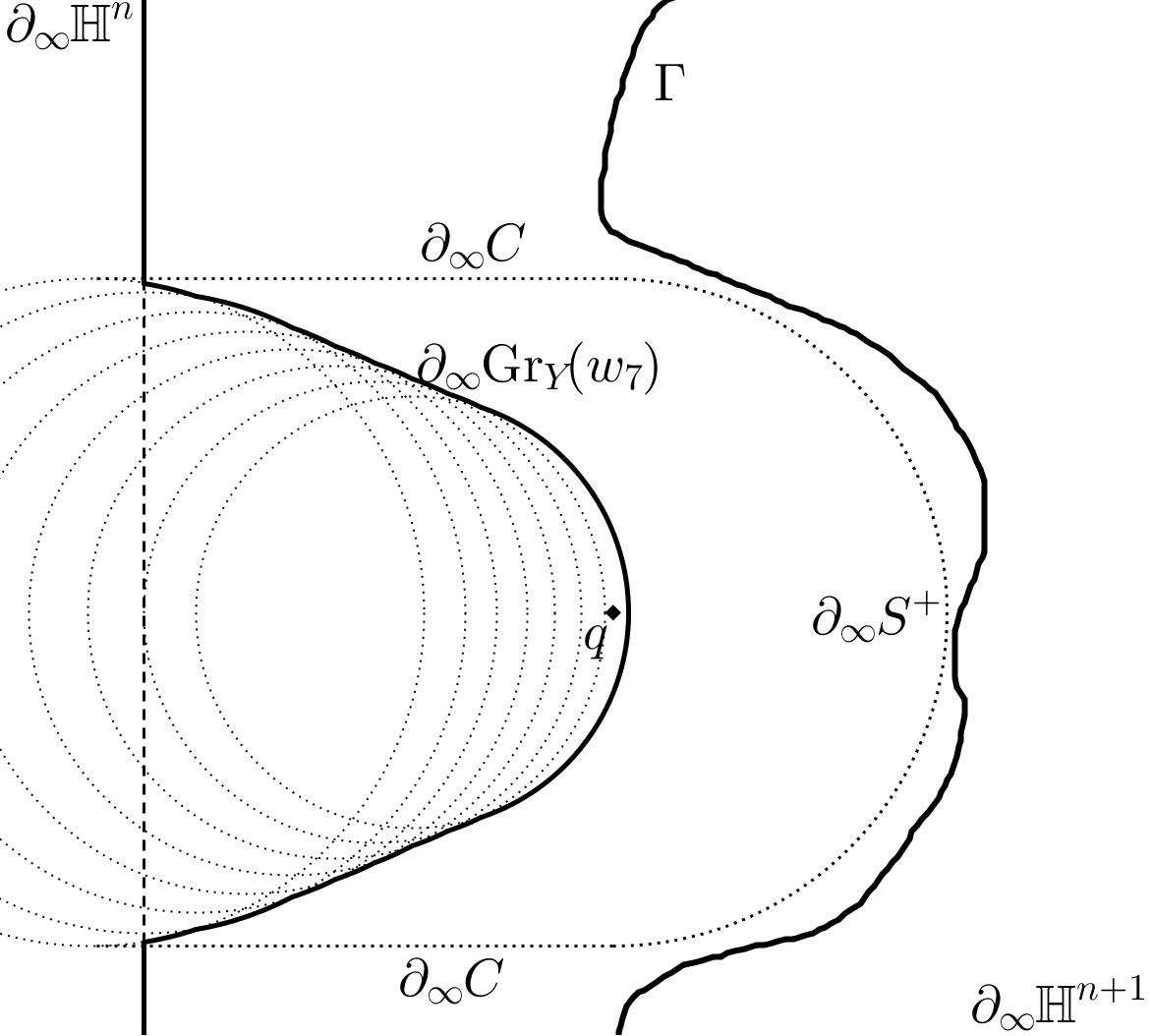}} \fbox{\includegraphics[scale=0.4]{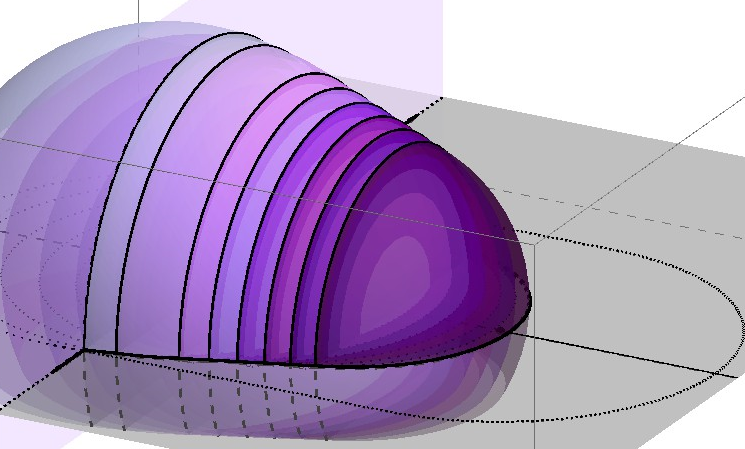}}}

\medskip

Since the Killing graph of $v_k$ is contained in $H_k$, that is obviously a minimal hypersurface of $\mathbb{H}^{n+1}$, we obtain that the following sequence of functions defined inductively is a sequence of elements of $S_{\phi}$:

$$ w_0(P) =   \left\{ \begin{array}{l} 
  \max\{v_0(P),0\}, \text{ if }P \in D_0,\\ 0,\text{ otherwise }.
\end{array}\right.$$

$$ \vdots$$

$$w_k(P) =   \left\{ \begin{array}{l} 
  \max\{v_k(P),w_{k-1}(P)\}, \text{ if }P \in D_k,\\ w_{k-1}(P),\text{ otherwise }.
\end{array}\right.$$

To finish, it suffices to prove that $\exists k \in \mathbb{N}$ such that $l+\varepsilon<t_k<l+1$. To prove it, notice that 

$$t_k = (1+R_1+\dots + R_{k-1})\sin \alpha - (1+R_1+\dots + R_k)\sin \beta=$$

$$ = \frac{1-\left(\frac{\cos \alpha}{\cos \beta}\right)^{k}}{1-\left(\frac{\cos \alpha}{\cos \beta}\right)}\sin \alpha  -\frac{1-\left(\frac{\cos \alpha}{\cos \beta}\right)^{k+1}}{1-\left(\frac{\cos \alpha}{\cos \beta}\right)}\sin \beta,$$ and since $\cos \alpha < \cos \beta$,

$$\lim_{k\to+\infty} t_k = \frac{\sin \alpha - \sin \beta}{1-\cos\alpha/\cos \beta},$$ which have been choosen to satisfy \eqref{alpha}. \end{proof}

\medskip


\medskip

We now have all necessary elements to prove the main result.

\begin{proof}[Proof of Theorem \ref{par}]
Let $\mathbb{H}^{n}$ be the a totally geodesic hypersurface such that $\partial_{\infty}\mathbb{H}^{n}=E_{1}.$ Let $Y$ be a Killing field orthogonal to $\mathbb{H}^{n}$ which orbits are horocycles orthogonal to $\mathbb{H}^{n}$ and such that $Y\left(  p\right)  =0.$ Let $F$ be a hypersphere of $\mathbb{H}^{n+1}$ such that $\partial_{\infty}F=E_{2}$ and that has constant mean curvature $\left\vert H\right\vert $ with respect to a unit normal vector field $\eta$ such that $\left\langle \eta ,Y\right\rangle \leq0.$ 

The orbits of $Y$ on $\partial_{\infty}\mathbb{H}^{n+1}\backslash\left\{  p\right\}  $ are of the form $C\backslash\left\{p\right\}  $ where $C$ is a circle passing through $p$ and orthogonal to
$E_{1}$ and $E_{2}$\ at $p.$ Therefore, from the hypothesis it follows that $\Gamma$ is the $Y-$Killing graph of a function $\phi_{Y}\in C^{0}\left(\partial_{\infty}\mathbb{H}^{n}\backslash\left\{  p\right\}  \right)  .$ Also, $F$ is the graph of a positive solution $w_{Y}\in C^{\infty}\left(
\mathbb{H}^{n}\right)  \cap C^{0}\left(  \overline{\mathbb{H}^{n}}\backslash\left\{  p\right\}  \right)  $ of $Q_{H}=0$ in $\mathbb{H}^{n}.$

Denote by $\Psi$ the flow of $Y$ (note that since $Y(p)=0$, $\Psi(p,t)=p$ for all $t\in\mathbb{R}$).
Obviously $\sigma\equiv 0$ is a subsolution. Furthermore $\sigma|_{\partial_{\infty}\mathbb{H}^{n}}\leq\phi\leq w|_{\partial_{\infty}\mathbb{H}^{n}}$ and it follows that the function $u=u_{\phi}$ defined in Theorem \ref{perron} belongs to $C^{\infty}\left(  \mathbb{H}^{n}\right)  $ and satisfies $Q_{H}\left(u\right)  =0.$ 

To finish the proof, one needs to prove that $u$ extends continuously to $\partial_{\infty}%
\mathbb{H}^{n}\backslash\left\{  p\right\}  $ and that $u|_{\partial_{\infty
}\mathbb{H}^{n}\backslash\left\{ p\right\}  }=\phi$. 

Proposition \ref{lbarrier} imply that $u_{\ab \Hi^{n}}\le \phi,$ because it implies that there are elements of $S_{\phi}$ arbitrarily close to any point of $\Gamma$.

It is also important to construct upper barriers. This step is easier than the construction of lower barriers because we do not need to construct global functions. Indeed, given $q$ in the connected component of $\ab \Hi^{n+1}\setminus \Gamma$ that does not contains $E_1$, take a CMC $H$ hypersphere  $J$ which mean curvature vector points to $\Gamma$ (if $H=0$, the orientation is irrelevant) and such that $\ab J \cap \Gamma = \emptyset$. From the Comparison Principle it follows that ${\rm{Gr}}_Y(v)\cap J=\emptyset$ for any $v\in S_{\phi}$.\end{proof}

\subsection{Theorem \ref{gs} revisited}\label{gsrevisited}

Theorem \ref{gs} may be also written in intrinsic terms and with a Killing graph point of view:

\begin{theorem}
\label{hyp}Let $p_{1},p_{2}$ be two distinct points in the asymptotic boundary $\partial_{\infty}\mathbb{H}^{n+1}$ of $\mathbb{H}^{n+1}$ and $\Gamma \subset\partial_{\infty}\mathbb{H}^{n+1}$ be a compact embedded topological hypersurface not passing neither through $p_{1}$ nor $p_{2}$ and that intersects transversely any arc of circle of $\partial_{\infty}\mathbb{H}^{n+1}$ having $p_{1}$ and $p_{2}$ as ending points. Let $\left\vert H\right\vert <1$ be given. Then there exists an unique properly embedded, complete $C^{\infty}$ hypersurface $\Sigma$ of $\mathbb{H}^{n+1}$ with CMC $H$ such that $\partial_{\infty}\Sigma=\Gamma$ and $\overline{\Sigma}=\Sigma\cup\Gamma$ is an embedded compact topological hypersurface of $\overline{\mathbb{H}}^{n+1}.$

Moreover there exist a totally geodesic hypersurface $\mathbb{H}^{n}$ of $\mathbb{H}^{n+1}$, a Killing field $X$ which orbits are hypercycles orthogonal to $\mathbb{H}^{n}$ and a function $u\in C^{\infty}\left(\mathbb{H}^{n}\right)  \cap C^{0}\left(  \overline{\mathbb{H}}^{n}\right)  $ such that $\overline{\Sigma}={\rm{Gr}}_{X}(u)$.
\end{theorem}

We now present the sketch of a proof of Theorem \ref{hyp} using the same techniques of Theorem \ref{par}.

First of all, let $\gamma$ be the geodesic of $\mathbb{H}^{n+1}$ such that $\partial_{\infty}\gamma=\left\{  p_{1},p_{2}\right\}  ,$ oriented from $p_{1}$ to $p_{2}.$ Let $X$ be the Killing field which orbits are equidistant hypercycles of $\gamma$ and assume that $X|_{\gamma}$ induces the same orientation of $\gamma.$  Let $\mathbb{H}^{n}$ be a totally geodesic hypersurface orthogonal to $\gamma$. 

To apply Perron's method in $\Hi^n$ it suffices to prove that there are sub and supersolutions $\sigma \le w$ of \eqref{pde} (with $Z=X$, see Section \ref{killing}) such that $\sigma_{\ab \Hi^n} \le \phi \le w_{\ab \Hi^n}.$
The subsolution is $\sigma\equiv 0$ (that is in fact a minimal hypersurface). To construct the supersolution, it is useful to work in the half-space model again: we take $p_{1}=0,$ $p_{2}=\left\{  x_{n+1}=\infty\right\}  $ and $\mathbb{H}^{n}$ as the half-sphere centered at $0$ with Euclidean radius $1.$ Then, since $\Gamma\subset\partial_{\infty}\mathbb{H}$ is compact and $p_{2}\notin\Gamma$ it follows that $\Gamma$ is compact in $\mathbb{R}^{n}:=\left\{  x_{n+1}=0\right\}  .$ Now, take a sphere $E$ of $\mathbb{R}^{n}$ centered at the origin $0$, containing $\Gamma$ in its interior, and choose $F$ as the hypersphere of $\mathbb{H}^{n+1}$ with CMC $\left\vert H\right\vert $ with respect to the unit normal vector field pointing to $0$ and having $E$ as asymptotic boundary. $F$ is the Killing graph of the desired function $w$.

To finish the proof, one needs to find barriers. The upper barriers are given by CMC $H$ hyperspheres as in the previous case. As to the lower barriers, the following picture illustrates the construction in this case:

\medskip

\centerline{\fbox{\includegraphics[scale=0.35]{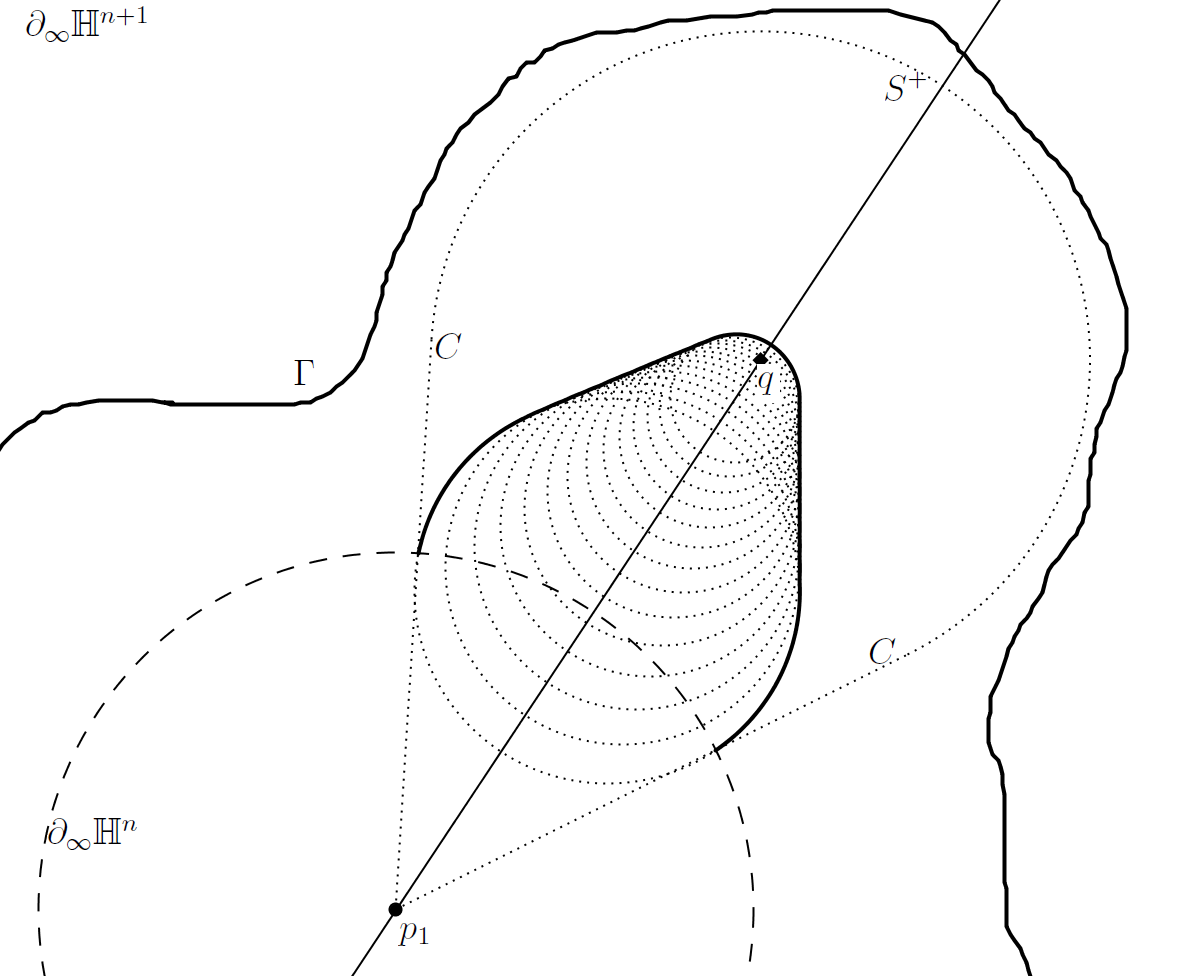}}}

\centerline{
\begin{tabular}{ccc}
{\small Jaime B. Ripoll} & $\hspace{1.5cm}$& {\small Miriam Telichevesky} \\
{\small Univ. Federal do Rio Grande do Sul}  &  & {\small Univ. Federal do Rio Grande do Sul}\\
{\small Instituto de Matem\'atica} &
& {\small Instituto de Matem\'atica}\\
{\small Av. Bento Gon\c calves 9500} & & {\small Av. Bento Gon\c calves 9500}\\
{\small 91540-000 Porto Alegre-RS } & &{\small 91540-000 Porto Alegre-RS }\\
{\small  BRASIL} &  & {\small BRASIL} \\
{\small jaime.ripoll@ufrgs.br}& &{\small miriam.telichevesky@ufrgs.br} \\
\end{tabular}}

\end{document}